\documentclass[12pt]{amsart}
\usepackage{amsthm,amsmath,amssymb,amsfonts}
\usepackage{enumitem}
\usepackage{hyperref}

\newcommand{\ga}{\gamma}

\newcommand{\ka}{\kappa}
\newcommand{\Ga}{\Gamma}

\newcommand{\M}{\mathcal{M}}
\newcommand{\calS}{\mathcal{S}}

\newcommand{\oM}{\overline{\mathcal{M}}}
\newcommand{\oMgn}{\overline{\mathcal{M}}_{g,n}}

\newcommand{\ZZ}{\mathbb{Z}}
\newcommand{\QQ}{\mathbb{Q}}

\DeclareMathOperator{\Aut}{Aut}
\newcommand{\D}{\mathcal{D}}
\renewcommand{\S}{\mathcal{S}}
\renewcommand{\d}{\partial}

\theoremstyle{plain}
\newtheorem{theorem}{Theorem}
\newtheorem{lemma}[theorem]{Lemma}
\newtheorem{proposition}[theorem]{Proposition}
\newtheorem{corollary}[theorem]{Corollary}

\theoremstyle{definition}
\newtheorem{definition}[theorem]{Definition}
\newtheorem{remark}[theorem]{Remark}
\newtheorem*{note}{Note}

\usepackage{color}

\begin{document}
\title{Topological recursion relations from Pixton's formula}

\author[E.~Clader]{Emily Clader}
\address{Department of Mathematics, San Francisco State University, San Francisco, CA 94132-1722, USA}
\email{eclader@sfsu.edu}
\author[F.~Janda]{Felix Janda}
\address{Department of Mathematics, University of Notre Dame, Notre Dame, IN 46556, USA}
\email{fjanda@nd.edu}
\author[X.~Wang]{Xin Wang}
\address{School of Mathematics, Shandong University, Jinan 250100, Shandong, China}
\email{wangxin2015@sdu.edu.cn}
\author[D.~Zakharov]{Dmitry Zakharov}
\address{Department of Mathematics, Central Michigan University, Mount Pleasant, MI 48849, USA}
\email{dvzakharov@gmail.com}


\begin{abstract}

We prove that every degree-$g$ polynomial in the $\psi$-classes on $\oMgn$ can be expressed as a sum of tautological classes supported on the boundary with no $\kappa$-classes.  Such equations, which we refer to as topological recursion relations, can be used to deduce universal equations for the Gromov--Witten invariants of any target.

\end{abstract}

\maketitle

\section{Introduction}

The tautological rings are $\QQ$-subalgebras of the Chow ring of the moduli spaces of curves,
\[R^*(\oM_{g,n}) \subseteq A^*(\oM_{g,n}),\]
defined as the minimal system of such subalgebras closed under pushforward by the gluing morphisms
\[\oM_{g_1,n_1+1} \times \oM_{g_2,n_2+1} \rightarrow \oM_{g_1+g_2, n_1+n_2},\]
\[\oM_{g,n+2} \rightarrow \oM_{g+1,n}\]
and the forgetful morphisms
\[\oM_{g,n+1} \rightarrow \oM_{g,n}.\]
This elegant definition is due to Faber and Pandharipande \cite{2005FaberPandharipande}, who also proved that $R^*(\oM_{g,n})$ admits an explicit set of additive generators that we call {\it basic classes}.  The basic classes are constructed in terms of the $\psi$-classes
\[\psi_i=c_1(s_i^*(\omega_{\pi})), \;\;\; i=1, \ldots, n\]
and the $\kappa$-classes
\[\kappa_d = \pi_*(\psi_{n+1}^{d+1}), \;\;\; d \geq 0\]
on the boundary strata, where $\pi: \oM_{g,n+1} \rightarrow \oM_{g,n}$ forgets the last marked point and $s_i$ is the section of $\pi$ given by the $i$th marked point.

Products of basic classes are described by an explicit multiplication rule, so the study of the tautological ring amounts to a search for {\it tautological relations}: linear equations satisfied by the basic classes.  Of particular interest in our work are the equations in which no $\kappa$-classes appear, which we refer to as {\it topological recursion relations}, or TRRs.  The well-known WDVV equations, for example, are TRRs in genus zero, and other specific examples in genus $g \leq 4$ were discovered by Getzler \cite{Ge1,Ge2}, Belorousski--Pandharipande \cite{BP}, Kimura--Liu \cite{KL1, KL2}, and the third author \cite{Wang}.  Liu--Pandharipande \cite{2011LiuPandharipande}, in addition, proved TRRs expressing $\psi_1^k$ on $\oM_{g,1}$ in terms of $\kappa$-free boundary classes whenever $k \geq 2g$.

The main theorem of our work is the following:

\begin{theorem}
\label{thm:main}
For any $g$ and $n$ (such that $2g-2+n>0$), there exists a topological recursion relation for every degree-$g$ monomial in the $\psi$-classes on $\oM_{g,n}$.
\end{theorem}

More explicitly, a topological recursion relation for $\alpha = \psi_1^{l_1} \cdots \psi_n^{l_n}$ refers to a tautological relation expressing $\alpha$ in terms of basic classes supported on the boundary of $\oM_{g,n}$ with no $\kappa$-classes; a precise definition is given in Section~\ref{sec:strata} below.

This theorem can be seen as a next step in a series of vanishing results for the tautological ring. Looijenga~\cite{1995Looijenga} proved that the tautological ring $R^*(\M_g)$, which is generated by monomials in $\kappa$-classes, vanishes in degrees greater than or equal to $g-1$. Ionel proved in~\cite{2002Ionel} that $R^*(\M_{g,n})$, which is additively generated by monomials in the $\psi$- and $\kappa$-classes, vanishes (in cohomology) for degrees greater than $g-1$. The results of Graber--Vakil~\cite{2005GraberVakil} and Faber--Pandharipande~\cite{2005FaberPandharipande} imply that any monomial in the $\psi$- and $\kappa$-classes of degree greater than or equal to $g$ can be expressed, in $R^*(\oM_{g,n})$, as a boundary class. This raises the natural question of finding explicit boundary formulas for such monomials on $\oM_{g,n}$.

In~\cite{2018CladerGrushevskyJandaZakharov}, the first, second and fourth author described an explicit algorithm for finding such formulas, thus reproving these vanishing results in an effective manner. We considered a family of relations in $R^*(\oM_{g,n})$, which are referred to as {\it Pixton's relations} in what follows and which arise from the study of the double ramification cycle on $\oM_{g,n}$. These relations were conjectured by Pixton and proven by the first and second authors in \cite{2016CladerJanda}, and are not to be confused with Pixton's $r$-spin relations (also conjectured by Pixton \cite{Pi12} and proved
  in~\cite{2015PandharipandePixtonZvonkine, J16}). We observed in~\cite{2018CladerGrushevskyJandaZakharov} that Pixton's relations are a natural candidate for producing TRRs, for two reasons. First, Pixton's relations start in degree $g+1$, which is only one degree greater than the TRRs that we seek. Second, Pixton's relations do not involve any $\kappa$-classes to begin with. To obtain relations in degree $g$, we multiply Pixton's relations by appropriate $\psi$-classes and push them forward under forgetful maps. We make the non-boundary contributions to these relations explicit in a family of cases.  From here, simple linear algebra shows that the resulting equations are sufficient to produce TRRs for any $\psi$-monomial. Thus, the results of this paper are a sharpening of the algorithm described in~\cite{2018CladerGrushevskyJandaZakharov}. 

One reason that TRRs are worthy of special interest is that they can be translated into universal equations for the descendent Gromov--Witten invariants of any target.  For targets with semisimple quantum cohomology, these universal equations are known to be sufficient to determine all Gromov--Witten invariants in genus one or two in terms of genus-zero data; this follows from work of Dubrovin--Zhang \cite{1998DubrovinZhang} in genus one and of Liu \cite{2002Liu} in genus two. Furthermore, it was conjectured in \cite{2006Liu} this should be true in higher genus, that is, all the higher genus Gromov-Witten invariants can be solved explicitly  from universal equations, if the quantum cohomology is semisimple. The analogue in genus three does not follow from the currently known universal equations, though, and very little is understood in any genus outside of the semisimple setting.  One might hope that a more complete picture of the TRRs on $\oM_{g,n}$ would fill some of these gaps.  Thus, we emphasize that the following corollary is a direct consequence of Theorem~\ref{thm:main}.

\begin{corollary}
Let $X$ be a nonsingular projective variety. All of the genus-$g$ descendent
Gromov--Witten invariants of $X$ can be reconstructed from the invariants 
\[\langle\tau_{k_1}(\phi_{1})\ldots \tau_{k_n}(\phi_n)\rangle_{g,n,\beta}^{X}\]
for which $\sum_{i=1}^{n}k_i<g+\delta_{g}^{0}$.
\end{corollary}

In particular, this gives an affirmative answer to a conjecture proposed by Faber and
Pandharipande in 2003 (see
\cite[Conjeture~3]{2005FaberPandharipande}).

\begin{note}
  At the final stages of writing this article, the
  paper~\cite{2018ShadrinEtAl} was posted.
  The authors of~\cite{2018ShadrinEtAl} studied the tautological rings
  $R^*(\oM_{g,n})$ using Pixton's
  $r$-spin relations.
  Although it is not specifically noted in~\cite{2018ShadrinEtAl}, our
  main result, Theorem~\ref{thm:main}, follows
  from~\cite[Lemma~5.2]{2018ShadrinEtAl}.
  Indeed, Lemma~5.2 reproves the vanishing in degrees at least $g$ of
  the tautological ring of $\M_{g,n}$, which is generated by monomials
  in $\psi$-classes.
  Specifically, it is proven using linear combinations of the
  polynomial $r$-spin relations \cite{2015PandharipandePixtonZvonkine}
  under the substitution $r=\frac 12$.
  As the authors note \cite[Section~2.4]{2018ShadrinEtAl}, ``graphs
  without dilaton leaves do not contribute to the tautological
  relations'', which is another way of saying that the resulting
  relations do not involve $\kappa$-markings and are thus topological
  recursion relations.

We do not know the relationship between the explicit topological recursion relations established in~\cite{2018ShadrinEtAl} and those established in our paper. In view of the continued interest in this problem, we believe that both sets of relations deserve further investigation.
\end{note}

\subsection{Plan of the paper}

Section~\ref{sec:preliminaries} contains the relevant preliminaries on the strata algebra and Pixton's relations.  In Section \ref{sec:results}, we describe a family of topological recursion relations constructed from Pixton's relations, make their contributions away from the boundary explicit, and use them to prove Theorem~\ref{thm:main}.

\subsection{Acknowledgments}
The authors would like to thank Samuel Grushevsky, Xiaobo Liu, Aaron Pixton, and Dustin Ross for useful discussions and inspiration.
We specially thank an anonymous referee for their careful reading and very valuable advice for improving our results.
The first author was partially supported by NSF DMS grant 1810969, the second author was partially supported by the CNRS and NSF DMS grant 2054830, and the third author was partially supported by NSFC grant 12071255 and SPNSF grant ZR2021MA101.

\section{The tautological ring and Pixton's formula}
\label{sec:preliminaries}

The additive generators of the tautological ring are expressed in terms of the strata algebra on $\oM_{g,n}$.  We recall the necessary definitions, referring the reader to \cite{2003GraberPandharipande} and \cite{2013Pixton} for further details.

\subsection{The strata algebra and the tautological ring of \texorpdfstring{$\oMgn$}{Mgn}}
\label{sec:strata}

A {\it stable graph} $\Ga=(V,H,E,L,g,p,\iota,m)$ of genus $g$ with $n$ legs consists of the following data:
\begin{enumerate}
\item a finite set of vertices $V$ with a genus function $g:V\to \ZZ_{\geq 0}$;

\item a finite set of half-edges $H$ with a vertex assignment $p:H\to V$ and an involution $\iota:H\to H$;

\item a set of edges $E$, which is the set of two-point orbits of $\iota$;

\item a set of legs $L$, which is the set of fixed points of $\iota$, and which is marked by a bijection $m:\{1,\ldots,n\}\to L$.

\end{enumerate}
We require that the following properties are satisfied:

\begin{enumerate}[label=(\roman*)]

\item The graph $(V,E)$ is connected.

\item For every vertex $v\in V$, we have
\begin{equation*}
  2g(v)-2+n(v) > 0,
\end{equation*}
where $n(v)=\#  p^{-1}(v)$ is the {\it valence} of the vertex $v$
\item The genus of the graph is $g$, in the sense that
\begin{equation*}\label{eq:genus}
g=h^1(\Ga)+\sum_{v\in V}g(v),
\end{equation*}
where $h^1(\Ga)=\# E-\# V+1$.
\end{enumerate}

An {\it automorphism} of a stable graph $\Ga$ consists of permutations of the sets $V$ and $H$ that commute with the maps $g$, $p$, $\iota$ and $m$ (and hence preserve $L$ and $E$). We denote by $\Aut(\Ga)$ the group of automorphisms of $\Ga$.

Given a stable curve $C$ of genus $g$ with $n$ marked points, its dual graph is a stable graph of genus $g$ with $n$ legs. If $\Ga$ is such a stable graph, let
\begin{equation*}
\oM_{\Ga}:=\prod_{v\in V}\oM_{g(v),n(v)}.
\end{equation*}
There is a canonical gluing morphism
\begin{equation}\label{eq:xi}
\xi_{\Ga}:\oM_{\Ga}\to \oMgn,
\end{equation}
whose image is the locus in $\oMgn$ having generic point corresponding to a curve with stable graph $\Ga$. The degree of $\xi_{\Ga}$, as a map of Deligne--Mumford stacks, is equal to $\#\Aut(\Ga)$.

Additive generators of the tautological ring can be described in terms of certain decorations on stable graphs $\Gamma$.  Namely, let
\begin{equation}
\label{eq:gamma}
\ga=(x_i:V\to \ZZ_{\geq 0},y:H\to \ZZ_{\geq 0})
\end{equation}
be a collection of functions such that
\begin{equation}
d(\ga_v)=\sum_{i>0}ix_i[v]+\sum_{h\in p^{-1}(v)} y[h]\leq 3g(v)-3+n(v)
\label{eq:degreecondition}
\end{equation}
for all $v\in V$.  Then, for each $v$, define
\begin{equation*}
\ga_v=\prod_{i>0}\ka_i^{x_i[v]}\prod_{h\in p^{-1}(v)}
\psi_h^{y[h]}\in A^{d(\ga_v)}(\oM_{g(v),n(v)}).
\end{equation*}
Associated to any such choice of decorations $\gamma$, there is a {\it basic class} on $\oM_{\Gamma}$, also denoted by $\gamma$, defined by
\begin{equation*}
\ga=\prod_{v\in V}\ga_v \in A^{d(\ga)}(\oM_{\Ga}).
\end{equation*}
Here, the degree is $d(\ga):=\sum_{v\in V}d(\ga_v)$, and we abuse notation slightly by using a product over classes on the vertex moduli spaces to denote a class on $\oM_{\Ga}$.

The {\it strata algebra}, by definition, is the finite-dimensional $\QQ$-vector space spanned by isomorphism classes of pairs $[\Ga,\ga]$, where $\Ga$ is a stable graph of genus $g$ with $n$ legs and $\ga$ is a basic class on $\oM_{\Ga}$.  The product is defined by excess intersection theory (see~\cite{1999Faber}).  More precisely, for $[\Ga_1,\ga_1], [\Ga_2,\ga_2]\in \calS_{g,n}$, the fiber product of $\xi_{\Ga_1}$ and $\xi_{\Ga_2}$ over $\oMgn$ is a disjoint union of $\xi_{\Ga}$ over all graphs $\Ga$ having edge set $E=E_1\cup E_2$, such that $\Ga_1$ is obtained by contracting all edges outside of $E_1$ and $\Ga_2$ is obtained by contracting all edges outside of $E_2$. We then define
\begin{equation}
[\Ga_1,\ga_1]\cdot [\Ga_2,\ga_2]=\sum_{\Ga}[\Ga,\ga_1\ga_2\varepsilon_{\Ga}],
\label{eq:strataproduct}
\end{equation}
where the excess class is
\begin{equation*}
\varepsilon_{\Ga}=\prod_{(h,h')\in E_1\cap E_2}-(\psi_h+\psi_{h'}),
\end{equation*}
and we set $\ga_1\ga_2=0$ whenever the degree condition \eqref{eq:degreecondition} is violated.

Pushing forward elements of the strata algebra along the gluing maps \eqref{eq:xi} defines a ring homomorphism
\begin{align*}
&q:\calS_{g,n}\to A^*(\oMgn)\\
&q([\Ga,\ga])=\xi_{\Ga*}(\ga),
\end{align*}
and the image of $q$ is precisely the tautological ring $R^*(\oMgn)$.  Elements of $\S_{g,n}$ in the kernel of $q$ are referred to as {\it tautological relations}.

The strata algebra is filtered by degree, defined as
\begin{equation*}
\deg[\Ga,\ga]=|E|+d(\ga),
\end{equation*}
which corresponds under $q$ to codimension in the Chow ring.
The product \eqref{eq:strataproduct} preserves the degree, so $\calS_{g,n}$ is a graded ring:
\begin{equation*}
\calS_{g,n}=\bigoplus_{d=0}^{3g-3+n}\calS_{g,n}^d.
\end{equation*}

Let $\d \S_{g,n}$ denote the subalgebra of $\S_{g,n}$ spanned by classes $[\Gamma, \gamma]$ in which the graph $\Gamma$ has at least one edge, and let $\d^0 \S_{g,n} \subseteq \d \S_{g,n}$ denote the subalgebra spanned by such classes with no $\kappa$'s in $\gamma$--- that is, with $x_i[v]=0$ for each $i$ and $v$.

\begin{definition}
\label{def:TRR}
Let $\xi \in \S_{g,n}$.  Then a {\it topological recursion relation}, or TRR, for $\xi$ is a tautological relation
\[q(\xi + \omega) = 0\]
in which $\omega \in \d^0\S_{g,n}$.
\end{definition}

When $\Gamma$ consists of a single vertex $v$, we will typically denote $[\Gamma, \gamma]$ simply by $\gamma_v$.  In particular, Theorem~\ref{thm:main} concerns TRRs for the case where $\xi$ is a monomial in the $\psi$-classes.

\subsection{Pixton's formula}

Fix $g$ and $n$, and fix a collection of integers $A = (a_1, \ldots, a_n)$ such that $\sum_j a_j =0$.  In this subsection, we recall the definition of Pixton's class, which is an inhomogeneous element of $\S_{g,n}$ depending on the choice of $A$.

To do so, one must first define auxiliary classes $\widetilde{\D}_{g,n}^r$, for an additional integer parameter $r> 0$, as follows.  For a stable graph $\Gamma=(V,H,g,p,\iota)$ of genus $g$ with $n$ legs, a \emph{weighting modulo $r$} on $\Gamma$ is defined to be a map
\begin{equation*}
  w: H \to \{0, \dotsc, r - 1\}
\end{equation*}
satisfying three properties:
\begin{enumerate}
\item For any $i \in \{1, \dotsc, n\}$ corresponding to a leg $\ell_i$ of $\Gamma$, we have $w(\ell_i) \equiv a_i \pmod{r}$.
\item For any edge $e\in E$ corresponding to two half-edges $h, h'\in H$, we
  have $w(h) + w(h') \equiv 0 \pmod{r}$.
\item For any vertex $v \in V$, we have
  $\sum_{h\in p^{-1}(v)} w(h) \equiv 0 \pmod{r}$.
\end{enumerate}
Define $\widetilde{\D}_{g,n}^r$ to be the class
\begin{equation*}
  \sum_{\Gamma} \frac{1}{\#\Aut(\Gamma)} \frac 1{r^{h^1(\Gamma)}}  \sum_{\substack{w \text{ weighting }\\ \text{mod } r \text{ on }\Gamma}} [\Gamma, \gamma_w] \in \S_{g,n},
\end{equation*}
where
\begin{equation}
  \label{eq:pixton}
  \gamma_w = \prod_{i=1}^n e^{\frac 12 a_i^2\psi_i} \prod_{(h,h')\in E} \frac{1 - e^{-\frac 12 w(h)w(h')(\psi_h + \psi_{h'})}}{\psi_h + \psi_{h'}},
 \end{equation}
which can be viewed as a basic class with $x_i(v) = 0$ for all $i>0$ and all vertices $V$.

The class $\widetilde{\D}_{g,n}^r$ is a polynomial in
$r$ for $r \gg 0$ (see \cite[Appendix]{2017JandaPandharipandePixtonZvonkine}).  Pixton's class, then, is defined as the constant term of this polynomial in $r$.  Using the fact that
\[a_1 = -(a_2 + \cdots +a_n),\]
we can express Pixton's class in terms of the variables $a_2, \ldots, a_n$ alone, so we denote it by
\[\D_{g,n}(a_2, \ldots, a_n),\]
and we denote its component in degree $d$ by $\D^d_{g,n}(a_2, \ldots, a_n)$.

The key point, conjectured by Pixton and proved by the first and second authors, is the following:
\begin{theorem}[\cite{2016CladerJanda}]
\label{thm:DR}
For each $d > g$, $\D^d_{g,n}(a_2, \ldots, a_n)$ is a tautological relation.
\end{theorem}

We refer to these as ``Pixton's relations" in what follows.

\subsection{Polynomiality properties}
\label{subsec:poly}

In fact, Theorem~\ref{thm:DR} yields tautological relations in a simpler form than initially apparent, as a result of the following crucial result of Pixton:

\begin{theorem}[Pixton, \cite{PixDR2}]
\label{thm:poly}
The class $\D_{g,n}(a_2, \ldots, a_n)$ depends polynomially on $a_2, \ldots, a_n$.
\end{theorem}

In particular, the coefficient of any monomial $a_2^{b_2} \cdots a_n^{b_n}$ in the class $\D_{g,n}^d(a_2, \ldots, a_n)$ yields a tautological relation, for each $d > g$:
\begin{equation}
\label{eq:rels}
q\left( \bigg[\D_{g,n}^d(a_2, \ldots, a_n)\bigg]_{a_2^{b_2} \cdots a_n^{b_n}} \right) = 0.
\end{equation}
We will use this fact repeatedly in what follows.

\begin{remark}
  In fact, we only require the weaker fact that the restriction of $\D_{g,n}(a_2, \ldots, a_n)$ to the locus of curves of compact type (that is, curves whose associated stable graph is a tree) depends polynomially on $a_2, \ldots, a_n$.  This weaker version of polynomiality is essentially clear from the definition of $\D_{g,n}(a_2, \ldots, a_n)$, given that a tree $\Gamma$ admits a unique weighting mod $r$.

  To see that it is sufficient for the purposes of producing tautological relations, first, the compact-type restriction will be sufficient for studying the leading terms of the relations, while the remaining terms are limited to boundary contributions.
  Second, note that taking the coefficient of $a_2^{b_2} \cdots a_n^{b_n}$ in \eqref{eq:rels} can alternatively be expressed as taking a particular linear combination (related to difference operators) of relations $\D^d_{g,n}(a_2, \ldots, a_n)$ for specific $a_i$'s, and such linear combinations make sense even for dual graphs where polynomiality is not assumed.
\end{remark}

By the definition of Pixton's class, none of the relations \eqref{eq:rels} involve $\kappa$-classes.  Furthermore, the powers of the $\psi$-classes that appear are controlled by the monomial in question:

\begin{lemma}
\label{lem:deg}
The degree of Pixton's class in $\psi_i$ is bounded by half the degree in $a_i$.

More explicitly, suppose that the class
\[\bigg[\D_{g,n}^d(a_2, \ldots, a_n)\bigg]_{a_2^{b_2} \cdots a_n^{b_n}} \in \S_{g,n}^d\]
is expanded in the standard basis of the strata algebra, and let $[\Gamma, \gamma]$ be a basis element that appears with nonzero coefficient.  For each leg $l_i$ of $\Gamma$ corresponding to a marked point $i \in \{2, \ldots, n\}$, if $\gamma = (\{x_j\}, y)$ in the notation of \eqref{eq:gamma}, then
\[y(l_i) \leq \frac{b_i}{2}.\]
\begin{proof}
  This is straightforward from the definition of
  $\D_{g,n}^d(a_2, \ldots, a_n)$ and Theorem~\ref{thm:poly}.
\end{proof}
\end{lemma}

\section{Main results}
\label{sec:results}

Given that Pixton's relations \eqref{eq:rels} do not involve any $\kappa$-classes, they are a natural candidate for producing topological recursion relations.  Moreover, by multiplying the relations by appropriate $\psi$-classes, one can ensure that $\kappa$-classes do not arise even after pushforward under forgetful maps, thereby yielding a large family of TRRs.  The proof of Theorem~\ref{thm:main}, which we detail in this section, is an application of this idea.

\subsection{A family of topological recursion relations}

Fix a genus $g$ and a number of marked points $n$.  Let $M$ be a monomial of degree $D\leq 2g+1$ in the variables $a_2, \ldots, a_n$, and let
\[N:= n+2g+2-D.\]
Define
\[\Omega_{g,M}^{\text{pre}} \in \S^{g+1}_{g,N}\]
to be the coefficient of the monomial $M \cdot a_{n+1} \cdots a_N$ in Pixton's class $\D^{g+1}_{g,N}(a_2, \ldots, a_N)$.  Note that we use here the polynomiality discussed in Section~\ref{subsec:poly}.

Let
\[\Pi: \oM_{g,N} \rightarrow \oM_{g,n+1}\]
be the forgetful map, and define
\[\Omega_{g,M}:= \Pi_*(\Omega_{g,M}^{\text{pre}} \cdot \psi_{n+2} \cdots \psi_{N}) \in \S_{g,n+1}^{g+1},\]
where we use $\Pi_*$ to denote the induced map on strata algebras.

\begin{lemma}
\label{lem:TRRs}
For any choice of monomial $M$, the class $\Omega_{g,M}$ is a TRR.  Furthermore, if
\[\pi: \oM_{g,n+1} \rightarrow \oM_{g,n}\]
is the forgetful map and $\pi_*$ denotes the induced map on strata algebras, then $\pi_*(\Omega_{g,M})$ is still a TRR.
\begin{proof}
By the definition of $\D_{g,N}^{g+1}$, there are no $\kappa$-classes in $\Omega_{g,M}^{\text{pre}}$, and furthermore, by Lemma~\ref{lem:deg}, $\Omega_{g,M}^{\text{pre}}$ has degree zero in $\psi_{n+1}, \ldots, \psi_{N}$.  Multiplying $\Omega_{g,M}^{\text{pre}}$ by $\psi_{n+2} \cdots \psi_{N}$ kills the contribution from any dual graph containing a rational tail that is no longer stable after the pushforward $\Pi_*$.  Thus, the only dual graphs occurring in $\Omega_{g,M}^{\text{pre}}\cdot \psi_{n+2} \cdots \psi_{N}$ are dual graphs from $\oM_{g,n+1}$ (with additional legs inserted at the vertices), and each such dual graph comes with degree one in $\psi_{n+2}, \ldots, \psi_N$.  It follows from the dilaton equation that $\Omega_{g,M}$ has no $\kappa$-classes.

The above discussion also implies that $\Omega_{g,M}$ has degree zero in $\psi_{n+1}$.  Thus, by the string equation, pushing it forward under $\pi_*$ creates no $\kappa$-classes, and this proves the second claim.
\end{proof}
\end{lemma}

The $n=1$ case of Theorem~\ref{thm:main} is an immediate consequence of Lemma~\ref{lem:TRRs}:

\begin{theorem}
\label{thm:n=1}
There exists a TRR for $\psi_1^g \in \S_{g,1}^g$.
\begin{proof}
Take $n=1$ and $M = 1$ in the above.  As observed in the proof of Lemma~\ref{lem:TRRs}, the only dual graphs contributing to $\Omega_{g,1}^{\text{pre}}$ are dual graphs from $\oM_{g,2}$ with additional legs at the vertices.  Of these, all contribute to the boundary part of $\pi_*(\Omega_{g,1})$ except for the trivial dual graph $\Gamma_0$ and the dual graph of the boundary divisor $\delta_{0, \{1,2\}}$ parameterizing curves on which both marked points lie on a rational tail.  It is straightforward to compute that the contribution from the latter graph is zero, whereas the contribution from $\Gamma_0$ is a nonzero number $\gamma$.  Thus, dividing $\pi_*(\Omega_{g,1})$ by $\gamma$ gives the desired TRR.
\end{proof}
\end{theorem}

The proof of Theorem~\ref{thm:main} for $n >1$ follows the same template, except that the terms in $\pi_*(\Omega_{g,M})$ with a trivial dual graph come from terms in $\Omega_{g,M}$ of two different types, both of which contribute nontrivially: those with a trivial dual graph, and those with  the dual graph of one of the boundary divisors $\delta_{0, \{i,n+1\}}$ for $i \in \{2, \ldots, n\}$.  The next section computes these contributions explicitly.

\subsection{Graph contributions}

Throughout this subsection, we assume that $n \geq 2$.
Let $\Gamma_0$ denote the trivial dual graph in $\oM_{g,n+1}$, and for
each $i \in \{1, \ldots, n\}$, let $\Gamma_i$ denote the dual graph in
$\oM_{g,n+1}$ on which marked points $i$ and $n+1$ lie on a rational
tail, which corresponds to the boundary divisor
$\delta_{0, \{i,n+1\}}$.  Our calculation of the contributions of these dual graphs to $\Omega_{g,N}$ repeatedly uses that, for a single-variable polynomial $p(x)$ and variables $x_1, \ldots, x_m$, the following combinatorial identity holds:
\begin{equation}
  \label{eq:reduce-variables}
  [p(x_1 + \dotsb + x_m)]_{x_1 \cdot \dotsb \cdot x_m}
    = m! [p(x)]_{x^m}.
\end{equation}

We begin by calculating the contribution of the trivial graph.

\begin{lemma}
  \label{lem:Gamma0}
  The contribution of $\Gamma_0$ to $\Omega_{g,M}$ is
  \begin{equation*}
    \frac{(2g-2+N-1)! (N - n)!}{(2g-2+n)!} [P_0(a_2, \dotsc, a_n, x)]_{M \cdot x^{N - n}},
  \end{equation*}
  where $[P_0(a_2, \dotsc, a_n, x)]_{M \cdot x^{N - n}}$ denotes the
  coefficient of $M \cdot x^{N - n}$ in the power series
  \begin{equation*}
    P_0(a_2, \dotsc, a_n, x) = \exp\left(\frac{1}{2}\left(\sum_{i=2}^n a_i + x\right)^2 \psi_1 + \frac{1}{2}\sum_{i=2}^n a_i^2 \psi_i\right).
  \end{equation*}
\begin{proof}
  The contribution of $\Gamma_0$ to $\Omega_{g,M}^{\text{pre}}$ is
  given by equation \eqref{eq:pixton}.
  Namely, if we set
  \[M':= M \cdot a_{n+1} \cdots a_{N},\]
  then the contribution of $\Gamma_0$ to $\Omega_{g,M}^{\text{pre}}$ is
  \[\left[\exp\left(\frac{1}{2} \sum_{i=1}^N a_i^2 \psi_i\right)\right]_{M'}
      = \left[ \exp\left(\frac{1}{2}\left(\sum_{i=2}^N a_i\right)^2 \psi_1 + \frac{1}{2}\sum_{i=2}^n a_i^2 \psi_i\right)\right]_{M'},\]
  where we use that $a_1 = -(a_2 + \cdots + a_N)$, and the second sum
  is only up to $n$ since $M'$ is linear in the variables
  $a_{n+1} \ldots, a_N$.
  Applying the identity \eqref{eq:reduce-variables} to the
  variables $a_{n+1}, \dotsc, a_N$, we rewrite the contribution of
  $\Gamma_0$ to $\Omega_{g,M}^{\text{pre}}$ as
  \begin{equation}
  \label{eq:Gamma0pre}
    (N - n)! \left[ \exp\left(\frac{1}{2}\left(\sum_{i=2}^n a_i + x\right)^2 \psi_1 + \frac{1}{2}\sum_{i=2}^n a_i^2 \psi_i\right)\right]_{M \cdot x^{N - n}}
  \end{equation}
  Multiplying \eqref{eq:Gamma0pre} by $\psi_{n+2}\cdots \psi_N$ and pushing forward, repeated application of the dilaton equation gives a coefficient of
  \[\frac{(2g-2+N-1)!}{(2g-2+n)!},\]
  which proves the claim.
\end{proof}
\end{lemma}

Next, we compute the contribution of $\Gamma_i$ to $\Omega_{g,M}$ for each $1 \leq i \leq n$.  We denote by $\psi'$ the $\psi$-class on the half-edge of $\Gamma_i$ adjacent to the genus-$g$ vertex.

\begin{lemma}
  \label{lem:Gammai}
  The contribution of $\Gamma_1$ to $\Omega_{g,M}$ is zero.  For $2\leq i\leq n$,  the contribution of $\Gamma_i$ to $\Omega_{g,M}$ is
  \begin{multline*}
    -\frac{(2g + 1 - D)!}{(2g - 3 + n)!} \sum_{n_1 + n_2 = 2g+1-D} (2g-3+n+n_1)! (n_2 + 1)! \\
  \cdot [P_i(a_2, \dotsc, a_n, x_1, x_2)]_{M \cdot x_1^{n_1} x_2^{n_2 + 1}},
  \end{multline*}
  where
  \begin{multline*}
    P_i(a_2, \dotsc, a_n, x_1, x_2)
    = \sum_{\ell \geq 0} \frac{(a_i + x_2)^{2\ell+2}}{2^{\ell+1} (\ell+1)!} (\psi')^{\ell} \\
    \cdot  \exp\left(\frac{1}{2} \left( \sum_{j=2}^n a_j + x_1 + x_2\right)^2 \psi_1 + \frac{1}{2}\sum_{\substack{j \in \{2, \ldots, n\}\\ j \neq i}} a_j^2 \psi_j\right).
  \end{multline*}
\begin{proof}
For each choice of partition $\{n+2, \ldots, N\} = I_1 \sqcup I_2$, there is a graph $\widetilde{\Gamma}_i^{I_1, I_2}$ on $\oM_{g,N}$ whose image under the forgetful map $\Pi$ is $\Gamma_i$; namely, $I_1$ gives the additional legs on the genus-$g$ vertex of $\Gamma_i$ and $I_2$ gives the additional legs on the genus-$0$ vertex.  By the same reasoning as in the proof of Theorem~\ref{thm:n=1}, these are the only dual graphs mapping to $\Gamma_i$ under $\Pi$ that have nonzero contribution to $\Omega_{g,M}$.  For each such dual graph, let $\psi'$ denote the $\psi$-class on the genus-$g$ vertex of and let $\psi''$ denote the $\psi$-class on the genus-$0$ vertex.

Denote by $\text{Contr}(\widetilde{\Gamma}_i^{I_1,I_2})$ the contribution of $\widetilde{\Gamma}_i^{I_1, I_2}$ to $\Omega_{g,M}^{\text{pre}}$.  When we multiply $\Omega_{g,M}^{\text{pre}}$ by
$\psi_{n+2} \cdots \psi_N$, any appearance of $\psi''$, $\psi_i$, or
$\psi_{n+1}$ in $\text{Contr}(\widetilde{\Gamma}_i^{I_1,I_2})$ is
killed for dimension reasons.
Thus, it suffices to compute
$\text{Contr}(\widetilde{\Gamma}_i^{I_1,I_2})\bigg|_{\psi'',\psi_i,\psi_{n+1}=0}$.  If $i > 1$, then from \eqref{eq:pixton} we see that this contribution equals
\begin{multline*}
  -\left[ \exp\left(\frac{1}{2} \left( \sum_{j=2}^N a_j\right)^2 \psi_1 + \frac{1}{2}\sum_{\substack{j \in \{2, \ldots, n\}\\ j \neq i}} a_j^2 \psi_j\right)\right.\\
  \cdot \left.\sum_{\ell \geq 0} \frac{(a_i + a_{n+1} + a_{I_2})^{2\ell+2}}{2^{\ell+1} (\ell+1)!} (\psi')^{\ell} \right]_{M'},
\end{multline*}
in which we denote $M' = M \cdot a_{n+1} \cdots a_N$ (as above), and
\begin{equation*}
  a_I := \sum_{j \in I} a_j.
\end{equation*}
Applying the identity \eqref{eq:reduce-variables} twice, we can rewrite the
above as
\begin{equation*}
  - (\#I_1)! (\#I_2 + 1)! [P_i(a_2, \dotsc, a_n, x_1, x_2)]_{M \cdot x_1^{\#I_1} x_2^{\#I_2 + 1}}
\end{equation*}
In particular, this depends only on $n_1 := \#I_1$ and
$n_2 := \#I_2$.
It follows that, when $i >1$, the contribution of $\Gamma_i$ to $\Omega_{g,M}$
equals
\begin{multline*}
  -\sum_{n_1 + n_2 = 2g+1-D} \binom{2g+1-D}{n_1}\cdot \frac{(2g-3+n+n_1)!}{(2g-3+n)!} \\
  \cdot \frac{n_2!}{0!} \cdot n_1! (n_2 + 1)! [P_i(a_2, \dotsc, a_n, x_1, x_2)]_{M \cdot x_1^{n_1} x_2^{n_2 + 1}},
\end{multline*}
in which the two quotients of factorials come from the dilaton
equation on the genus-$g$ and genus-$0$ vertices, respectively.

This proves the lemma in the case when $i>1$.  When $i=1$, on the other hand, \eqref{eq:pixton} shows that $\text{Contr}(\widetilde{\Gamma}_i^{I_1,I_2})\bigg|_{\psi'',\psi_i,\psi_{n+1}=0}$ equals
\[   -\left[ \exp\left( \frac{1}{2}\sum_{j=2}^{n} a_j^2 \psi_j\right)
  \cdot \sum_{\ell \geq 0} \frac{(a_2+...+a_{n}+a_{I_1})^{2\ell+2}}{2^{\ell+1} (\ell+1)!} (\psi')^{\ell} \right]_{M'},\]
 which is manifestly equal to zero since the variable $a_{n+1}$ does not appear.
\end{proof}
\end{lemma}

\subsection{Proof of Theorem~\ref{thm:main}}
\label{subsec:proof}

The sum of the results of Lemma~\ref{lem:Gamma0} and Lemma~\ref{lem:Gammai}, after pushing forward under $\pi: \oM_{g,n+1} \rightarrow \oM_{g,n}$, gives precisely the contribution of the trivial graph to a TRR in $A^g(\oM_{g,n})$.  There is one such TRR for each choice of the monomial $M$, and our goal is to use these to deduce TRRs for any monomial
\begin{equation}
\label{eq:psimon}
\psi_1^k \prod_{j=2}^n \psi_i^{l_j}
\end{equation}
of degree $g$ in the $\psi$-classes on $\oM_{g,n}$, by induction on
$n$ and $l_2, \dotsc, l_n$.  It is natural, toward this end, to consider the TRR associated to
\[M = \prod_{j=2}^n a_j^{2l_j}.\]
The power of $\psi_j$ for $j \ge 2$ in the contribution of $\Gamma_0$
to $\Omega_{g,M}$ (calculated in Lemma~\ref{lem:Gamma0}) is bounded by
$l_i$, so after applying the string equation, the same holds for the
pushforward of this contribution under $\pi$.

The problem, however, is that there is no constraint on the power of
$\psi'$ in the contribution of $\Gamma_i$ calculated in
Lemma~\ref{lem:Gammai}, so the power of $\psi_i$ in the push-forward
of that contribution under $\pi$ can be $l_i$ or above.
Below, we make the surprising observation that, for $l_i > 0$, these
problematic terms cancel each other in the combination
\begin{multline}
  \label{eq:cancel-Gamma_i}
  \frac{2l_i}{(2g + 1 - \sum_{j=2}^{n} 2l_j)!} \pi_* \Omega_{g, a_2^{2l_2} \cdot \dotsb \cdot a_n^{2l_n}} \\
  - \frac 1{(2g + 2 - \sum_{j=2}^{n} 2l_j)!} \pi_* \Omega_{g, a_2^{2l_2} \cdot \dotsb \cdot a_{i - 1}^{2l_{i - 1}} a_i^{2l_i - 1} a_{i + 1}^{2l_{i + 1}} \cdot \dotsb \cdot a_n^{2l_n}}
\end{multline}
of TRRs.
This cancellation is key to setting up an induction on the $l_j$.

\begin{lemma}
  \label{lem:cancel-Gamma_i}
  In the contribution of $\Gamma_i$ to the combination
  \eqref{eq:cancel-Gamma_i}, the power of $\psi_j$ for $j \ge 2$ is
  bounded by $l_j$ if $j \neq i$, and by $l_i - 1$ if $j = i$.
\end{lemma}
\begin{proof}
  The statement for $j \neq i$ follows directly from
  Lemma~\ref{lem:Gammai}, and therefore, we will assume $j = i$ from
  now on.
  Furthermore, we may assume $i = 2$ without loss of generality.
  By Lemma~\ref{lem:Gammai}, the contribution of $\Gamma_2$ to
  \begin{equation*}
    \frac{2l_2}{(2g + 1 - \sum_j 2l_j)!}  \cdot \Omega_{g, a_2^{2l_2} \cdot \dotsb \cdot a_n^{2l_n}}
  \end{equation*}
  is
  \begin{multline*}
    -\frac{2l_2}{(2g - 3 + n)!} \sum_{n_1 + n_2 = 2g+1-\sum_j 2l_j} (2g-3+n+n_1)! (n_2 + 1)! \\
    \cdot [P_i(a_2, \dotsc, a_n, x_1, x_2)]_{a_2^{2l_2} \cdot \dotsb \cdot a_n^{2l_n} \cdot x_1^{n_1} x_2^{n_2 + 1}}.
  \end{multline*}
  Noting that for any polynomial $f(z)$, we have
  \begin{equation*}
    [f(a_2 + x_2)]_{a_2^{2l_2} x_2^{n_2 + 1}}
    = \frac{n_2 + 2}{2l_2} [f(a_2 + x_2)]_{a_2^{2l_2 - 1} x_2^{n_2 + 2}},
  \end{equation*}
  we may rewrite the contribution as
  \begin{multline*}
    -\frac 1{(2g - 3 + n)!} \sum_{n_1 + n_2 = 2g+1-\sum_j 2l_j} (2g-3+n+n_1)! (n_2 + 2)! \\
    \cdot [P_i(a_2, \dotsc, a_n, x_1, x_2)]_{a_2^{2l_2 - 1} \cdot \dotsb \cdot a_n^{2l_n} \cdot x_1^{n_1} x_2^{n_2 + 2}}.
  \end{multline*}
  On the other hand, the contribution of $\Gamma_2$ to
  \begin{equation*}
    \frac{1}{(2g+2-\sum_{j=2}^n 2 l_j)!}\Omega_{g, a_2^{2l_2 - 1}a_3^{2l_3} \cdot \dotsb \cdot a_n^{2l_n}}
  \end{equation*}
  is
  \begin{multline*}
    -\frac 1{(2g - 3 + n)!} \sum_{n_1 + n_2 = 2g+2-\sum_j 2l_j} (2g-3+n+n_1)! (n_2 + 1)! \\
    \cdot [P_i(a_2, \dotsc, a_n, x_1, x_2)]_{a_2^{2l_2 - 1} \cdot \dotsb \cdot a_n^{2l_n} \cdot x_1^{n_1} x_2^{n_2 + 1}}.
  \end{multline*}
  By reindexing $n_2$, we see that these agree up to the boundary term
  \begin{multline*}
    -\frac 1{(2g - 3 + n)!} (2g-3+n+2g+2-\textstyle\sum_j 2l_j)! \\
    \cdot [P_i(a_2, \dotsc, a_n, x_1, x_2)]_{a_2^{2l_2 - 1} \cdot \dotsb \cdot a_n^{2l_n} \cdot x_1^{2g+2-\sum_j 2l_j} x_2}.
  \end{multline*}
  Inspecting the definition of $P_i$, the power of $\psi'$ in this
  term is bounded by $l_2 - 1$, and so also the power of $\psi_2$ in
  \eqref{eq:cancel-Gamma_i} is bounded by $l_2 - 1$.
\end{proof}

To achieve the same cancellation for all $i$, we consider the linear
combination
\begin{equation}
  \label{eq:cancellation}
  \sum_{d_2, \dotsc, d_n = 0}^1
  \frac{\prod_{j = 2}^n (-2l_j)^{1 - d_j}}{(2g + 1 - \sum_{j=2}^{n} (2l_j - d_j))!} \cdot \pi_* \Omega_{g, a_2^{2l_2 - d_2} \cdot \dotsb \cdot a_n^{2l_n - d_n}},
\end{equation}
defined when $l_j \ge 1$ for all $j$.  This is a TRR (since each $\pi_*\Omega_{g,M}$ is), and in the next proposition, we prove that it expresses \eqref{eq:psimon} in terms of ``lower" $\psi$-monomials modulo boundary, where a $\psi$-monomial $P$ is said to be \emph{lower} than \eqref{eq:psimon} if $\deg_{\psi_j}(P) \le l_j$ for $j \ge 2$ with at
least one of these inequalities being strict.

\begin{proposition}
  \label{prop:induction-step}
  Assume that $l_j \ge 1$ for all $j$.
  Then the TRR \eqref{eq:cancellation} expresses the monomial
  \eqref{eq:psimon} in terms of lower $\psi$-monomials and boundary
  terms.
\end{proposition}
\begin{proof}
  As remarked at the beginning of Section~\ref{subsec:proof},
  Lemma~\ref{lem:Gamma0} shows that the contribution of $\Gamma_0$ to
  the main summand $d_2 = \dotsb = d_n = 0$ is a multiple of
  \eqref{eq:psimon} up to lower $\psi$-monomials.
  Furthermore, the coefficient of \eqref{eq:psimon} is clearly
  positive, and hence nonzero.
  Similarly, the contribution of $\Gamma_0$ to the other summands in
  \eqref{eq:cancellation} consists only of lower $\psi$-monomials.

  Now, consider the contribution of one of the graphs $\Gamma_i$.
  Without loss of generality, we can assume that $i = 2$.
  In that case, Lemma~\ref{lem:cancel-Gamma_i} proves that the sum of
  the main term $d_2 = \dotsc = d_n = 0$ and the term $d_2 = 1$,
  $d_3 = \dotsb = d_n = 0$ is a combination of lower $\psi$-monomials.
  By the same proof, all other pairs of terms with fixed
  $d_3, \dotsc, d_n$ give a combination of lower $\psi$-monomials.
  Since the remaining dual graphs give only boundary contributions,
  this concludes the proof.
\end{proof}

This proposition sets up an induction that completes the proof of our main theorem: namely, that for any degree-$g$ $\psi$-monomial on $\oM_{g,n}$, there is a TRR expressing it in terms of boundary.

\begin{proof}[Proof of Theorem~\ref{thm:main}]
First, we use induction on $n$.  The base case is when $n=1$, in which case the result holds by Theorem~\ref{thm:n=1}.  Suppose, then, that the result holds on $\oM_{g,n-1}$, and consider a degree-$g$ $\psi$-monomial on $\oM_{g,n}$.  If some $\psi_j$ does not appear in this monomial, then modulo boundary, it is pulled back from $\oM_{g,n-1}$ and thus a TRR exists by induction.  Thus, it suffices to consider $\psi$-monomials of the form \eqref{eq:psimon} in which $l_j \geq 1$ for each $j$.  For these monomials, induction on the $l_j$ together with Proposition~\ref{prop:induction-step} completes the proof.
\end{proof}


\begin{thebibliography}{CMW11}

\bibitem[BP00]{BP}
P. Belorousski, R. Pandharipande.
\newblock  A descendent relation in genus 2.
\newblock {\em Ann. Scuola Norm. Sup. Pisa Cl. Sci.} 29, no. 1, 171-191, 2000.

\bibitem[CJ16]{2016CladerJanda}
E.~Clader and F.~Janda.
\newblock Pixton's double ramification cycle relations.
\newblock{\em Geom. Topo.} 22, no.2, 1069--1108, 2018.

\bibitem[CGJZ18]{2018CladerGrushevskyJandaZakharov}
E.~Clader, S.~Grushevsky, F.~Janda and D.~Zakharov.
\newblock Powers of the theta divisor and relations in the tautological ring.
\newblock International Mathematics Research Notices 2018 (24), 7725-7754, 2018

\bibitem[DZ98]{1998DubrovinZhang}
B.~Dubrovin and Y.~Zhang.
\newblock Bi-Hamiltonian hierarchies in $2$D topological field theory at one-loop approximation.
\newblock {\em Comm. Math Phys.} 198, no. 2, 311-361, 1998.

\bibitem[F99]{1999Faber}
C.~Faber.
\newblock Algorithms for computing intersection numbers on moduli spaces of curves, with an application to the class of the locus of Jacobians.
\newblock {\em New trends in algebraic geometry (Warwick, 1996)}, {\em London Math. Soc. Lecture Note Ser.} 264, 93-109.
\newblock Cambridge Univ. Press, Cambridge, 1999.

\bibitem[FP05]{2005FaberPandharipande}
C.~Faber and R.~Pandharipande.
\newblock Relative maps and tautological classes.
\newblock {\em J. Eur. Math. Soc.} 7, no. 1, 13-49, 2005.

\bibitem[Ge97]{Ge1}
E.~Getzler.
\newblock Intersection theory on  $\overline{\mathcal{M}}_{1,4}$ and elliptic Gromov-Witten invariants.
\newblock {\em J. Amer. Math. Soc.} 10, no. 4, 973-998, 1997.

\bibitem[Ge98]{Ge2}
 E.~Getzler.
 \newblock Topological recursion relations in genus 2.
 \newblock In {\em Integrable systems and algebraic geometry}, 73-106, 1998.

\bibitem[GraP03]{2003GraberPandharipande}
T.~Graber and R.~Pandharipande.
\newblock Constructions of nontautological classes on moduli spaces of curves.
\newblock {\em Michigan Math. J.} 51, no. 1, 93-109, 2003.

\bibitem[GV05]{2005GraberVakil}
T.~Graber and R.~Vakil. 
\newblock Relative virtual localization and vanishing of tautological classes on moduli spaces of curves.
\newblock {\em Duke Mathematical Journal} 130, no. 1, 1-37, 2005

\bibitem[I02]{2002Ionel}
E.-N. Ionel.
\newblock Topological recursive relations in $H_{2g}(\M_{g,n})$.
\newblock {\em Inventiones Mathematicae} 148, no. 3, 627-58, 2002

\bibitem[J16]{J16}
F.~Janda.
\newblock Relations on $\oM_{g,n}$ via equivariant Gromov--Witten theory of $\mathbb{P}^1$.
\newblock {\em Algebraic Geometry} (Foundation Compositio Mathematica) Volume 4, Issue 3, 311-336, 2017

\bibitem[JPPZ17]{2017JandaPandharipandePixtonZvonkine}
F.~Janda, R.~Pandharipande, A.~Pixton and D.~Zvonkine.
\newblock Double ramification cycles on the moduli spaces of curves.
\newblock{\em Publ. Math. Inst. Hautes
{\'E}tudes Sci.}, 125, no.1, 221--266, 2017.

\bibitem[KLLS18]{2018ShadrinEtAl}
R.~Kramer, F.~Labib, D.~Lewanski, S.~Shadrin.
\newblock The tautological ring of $\mathcal{M}_{g,n}$ via Pandharipande-Pixton-Zvonkine $r$-spin relations.
\newblock{\em Algebr. Geom.}, 5, no.6., 703-727, 2018

\bibitem[KL06]{KL1}
T.~Kimura, X.~Liu.
\newblock A genus-3 topological recursion relation.
\newblock {\em Comm. Math. Phys.}, 262, no.3, 645-661, 2006.

\bibitem[KL15]{KL2}
T.~Kimura, X.~Liu.
\newblock Topological recursion relations on $\overline{\mathcal{M}}_{3, 2}$.
\newblock {\em Sci. China Math.}, 58, no. 9,1909-1922, 2015.

\bibitem[L02]{2002Liu}
X.~Liu.
\newblock Quantum product on the big phase space and the Virasoro conjecture.
\newblock {\em Adv. Math.} 169, no. 2, 313-375, 2002.

\bibitem[L06]{2006Liu}
X.~Liu.
\newblock Gromov-Witten invariants and moduli spaces of curves.
\newblock In {\em International Congress of Mathematicians, Vol. II}, 791-812, 2006.

\bibitem[LP11]{2011LiuPandharipande}
X.~Liu, R.~Pandharipande.
\newblock New topological recursion relations.
\newblock {\em J. Algebraic Geom.} 20, no. 3, 479-494, 2011.

\bibitem[L95]{1995Looijenga}
E.~Looijenga.
\newblock On the tautological ring of $\M_g$.
\newblock {\em Inventiones Mathematicae} 121, no. 2 411-419, 1995.

\bibitem[PPZ15]{2015PandharipandePixtonZvonkine}
R.~Pandharipande, A.~Pixton and D.~Zvonkine.
\newblock Relations on $\overline{\mathcal{M}}_{g,n}$ via 3-spin structures.
\newblock {\em J. Amer. Math. Soc.} 28, 279-309, 2015.

\bibitem[Pi12]{Pi12}
A.~Pixton.
\newblock Conjectural relations in the tautological ring of $\overline{\mathcal{M}}_{g,n}$.
\newblock Preprint arXiv:1207.1918.

\bibitem[Pi13]{2013Pixton}
A.~Pixton.
\newblock The tautological ring of the moduli space of curves.
\newblock Ph.D. Thesis, Princeton University, 2013.

\bibitem[Pi17]{PixDR2}
A.~Pixton.
\newblock On combinatorial properties of the explicit expression for double ramification cycles.
\newblock In preparation.

\bibitem[W20]{Wang}
X.~Wang.
\newblock A genus-4 topological recursion relation for Gromov-Witten invariants.
\newblock {\em Sci. China Math.}, 63, no. 9, 101–112, 2020.


\end{thebibliography}
\end{document}